\documentclass[10pt]{article}
\usepackage{graphicx}
\usepackage{psfrag}
\usepackage{epsf}
\usepackage[cp1250]{inputenc}
\usepackage{amsmath,amsfonts,amssymb,latexsym}
\usepackage[left=2cm,top=2cm,right=2cm,bottom=4cm,nohead]{geometry}
\usepackage{setspace}
\usepackage{enumitem}
\usepackage{multirow}

\newtheorem{theorem}{Theorem}
\newtheorem{claim}{Claim}

\newtheorem{lemma}[theorem]{Lemma}

\newtheorem{example}{Example}
\def\vt{t\kern-0.22em\raise.18ex\hbox{\char'47}\lower.18ex\hbox{}\kern-0.08em}
\newcommand{\old}[1]{{}}
\renewcommand{\Pr}{{\rm Pr}}
\newcommand{\E}{{\rm E}}

\title{Abelian groups yield many large  families for the  diamond problem}

\author{\'Eva Czabarka \\
 Department of Mathematics, University of South Carolina\\
Columbia, SC 29208, USA\\
{\tt czabarka@math.sc.edu}
\and
Aaron Dutle\\ 
Department of Mathematics, University of South Carolina\\
Columbia, SC 29208, USA\\
{\tt dutle@mailbox.sc.edu}
\and 
Travis Johnston\\
Department of Mathematics, University of South Carolina\\
Columbia, SC 29208, USA\\
{\tt johnstjt@mailbox.sc.edu}
\and
L\'aszl\'o A. Sz\'ekely
\footnote{The first, second, and fourth authors acknowledge financial support from grant \#FA9550-12-1-0405 from the U.S.
Air Force Office of Scientific Research (AFOSR) and the Defense Advanced
Research Projects Agency (DARPA). The third and fourth authors  acknowledge financial support from the grant  \#1300547 from the U.S. National Science Foundation (NSF).  }\\
 Department of Mathematics, University of South Carolina\\
Columbia, SC 29208, USA\\
{\tt szekely@math.sc.edu}
 }


\def\FF{\mathcal{F}}
\def\BB{\mathcal{B}}

\date{\today}

\newcommand{\NN}{\mbox{$\mathbb N$}}

\newcommand{\ZZ}{\mbox{$\mathbb Z$}}

\newcommand{\eopf}{\raisebox{0.8ex}{\framebox{}}}

\newenvironment{proof}%
{\noindent{\bf Proof.}\ }%
{\hfill\eopf\par\bigskip}%

{\noindent{\bf Proof of #1.}\ }%
{\hfill\eopf\par\bigskip}%

{\noindent{\bf Sketch of a proof.}\ }%
{\hfill\eopf\par\bigskip}%

{\noindent{\bf Proof in progress.}\ }%
{\hfill\eopf\par\bigskip}%
\begin{document}

\maketitle

\begin{abstract}  
 There is much recent interest  in excluded subposets. Given a fixed poset $P$, how many subsets of $[n]$ can
found without a copy of $P$ realized by the subset relation? The hardest and most intensely investigated
problem of this kind is when $P$ is a {\sl diamond}, i.e. the power set of a 2 element set. In this paper, we show infinitely many asymptotically tight constructions using random set families defined from posets 
based on Abelian groups.
They are provided by the convergence of Markov chains on
groups. Such constructions suggest that the  diamond problem is hard.
\end{abstract}

\eject

\section{Introduction}
This introduction largely follows the concise and accurate description of the background and history from  \cite{lu}. For posets $P = (P,\leq)$ and $P' = (P',\leq')$, we say $P'$ is a {\sl weak subposet} of
$P$ if there exists an injection $f : P' \rightarrow P$ that preserves the partial ordering,
meaning that whenever $u \leq' v$ in $P'$, we have $f(u) \leq  f(v)$ in $P$ (see \cite{stanley}).
By  subposet we always mean weak subposet. The
{\sl height} $h(P)$ of  a poset $P$ is the length of the longest chain in $P$.
We consider a family $\FF$ of subsets of $[n]$ a poset for the subset relation. If $P$ is not a subposet of  $ \FF$, we say $\FF$ is $P$-free. We are interested in determining
the largest size of a $P$-free family of subsets of $[n]$, denoted $La(n, P)$. Let $P_k$ denote the  total order of $k$ elements that we term as  $k$-chain. 
The archetypal result is Sperner's Theorem \cite{sperner,engelbook}:   $La(n,P_2) =\binom{n}{\lfloor n/2\rfloor}$.
Let $\BB(n, k)$ denote  the middle $k$ levels in the subset  lattice of $[n]$ and
let $\sum(n, k) := |\BB(n, k)|$. Erd\H os \cite{erdos,engelbook} proved that $La(n,P_{k+1}) = \sum(n, k)$. 

For any $\FF$ family of subsets of $[n]$, define its {\sl Lubell function} $h_n(\FF) :=\sum_{F\in\FF} \frac{1}{\binom{n}{|F|}}$.
The celebrated
Bollob\'as--Lubell--Meshalkin--Yamamoto (BLYM) inequality asserts that for a $P_k$-free family  $\FF$, we have $h_n(\FF) \leq k-1$, which was originally shown for
$k=2$ 
\cite{bollobas, lubbell, meshalkin, yamamoto} and extended by     P. L. Erd\H os, Z. F\"uredi,  G.O.H. Katona \cite{all-2fur}. (For a generalization of the BLYM inequality, 
where cases of equality characterize mixed orthogonal arrays, see Aydinian, Czabarka and Sz\'ekely \cite{ayd}.) The BLYM inequality gives the book proof to 
$La(n,P_{k+1}) = \sum(n, k)$. In view of this, it makes sense to study
$\lambda_n(P) = \max{h_n(\FF)}$, where the maximization is over    $P$-free families $\FF$ in $[n]$.

G.O.H Katona had a key role starting the investigation  of  extremal problems with excluded
posets \cite{carrollkatona, annalisa, debonis,  tarjan, griggskat}. Katona and Tarj\'an   \cite{tarjan} obtained bounds on $La(n, V_2)$ and later  De Bonis and Katona  
\cite{annalisa} 
extended it to $La(n, V_r)$, where the $r$-fork $V_r$ is the poset $A < B_1, . . . ,B_r$. The answers are asymptotic to
$\binom{n}{\lfloor n/2\rfloor}$, and most of the work is devoted to finding second and third terms in the asymptotic expansion.

For excluded posets $P$ whose Hasse diagram is a tree, surpassing earlier results of Thanh \cite{thanh} and 
Griggs and Lu \cite{griggslu}, finally Bukh \cite{bukh} solved the asymptotic problem for the main term. 
For any poset $P$, define $e(P)$ to be the maximum $m$ such that for
all $n$, the union of the $m$ middle levels $\BB(n,m)$ does not contain $P$ as a
subposet.  The relevance of $e(P)$ was suggested by Mike Saks and Peter Winkler.
Bukh \cite{bukh} showed
\begin{equation} \label{fraction}
\pi(P)=:\lim_{n\rightarrow \infty}
\frac{La(n,P)}{\binom{n}{\lfloor n/2\rfloor}}
\end{equation}
 is $e(P)$.

   Katona \cite{katona130}  attributes the now famous {\sl diamond problem} to a question 
   of an unidentified member of the audience at his talk. After all, if trees are solved, the next
   open problems must allow some cycles in the Hasse diagram of $P$.
   The {\sl diamond} $D_k$ is defined as $A < B_1, . . . ,B_k < C$, and with the term diamond
   we normally refer to $D_2$.

Griggs and Li \cite{griggsli} introduced a relevant class of posets.
 They termed a  poset $P$  {\sl uniform-L-bounded}, 
 if $\lambda_n(P) \leq e(P)$ for all $n$. For any  uniform-L-bounded  posets $P$, Griggs and Li \cite{griggsli}  proved $La(n, P) =
\sum(n, e(P))$.  Griggs,
Li, and Lu \cite{griggslilu} showed that 
the chain $P_k$, the  diamond $D_k$ if $2^{m-1}-1\leq k \leq 2^m-  \binom{m}{\lfloor m/2\rfloor}  -1$ (including the numbers  $k = 3, 4, 7, 8, 9, 15, 16, . . .$)
are uniform-L-bounded posets, and so are the harps $H(l_1, l_2, . . . , l_k)$ (consisting of chains $P_1, . . . ,P_k$ with their top elements
identified and their bottom elements identified, for $l_1 > l_2 > \cdots  > l_k$).
 
 Griggs and Lu \cite{griggslu}  conjecture that 
 for {\sl any poset} $P$,
  the limit in (\ref{fraction})
  exists and is an integer. All the results above are compatible with this conjecture.
  
  The {\sl crown} ${\cal O}_{2t}$ is defined as $A_1<B_1>A_2<B_2> \cdots >A_t<B_t>A_1$. Crowns are neither trees nor uniform-L-bounded. 
  The crown ${\cal O}_{4}$ is known as the {\sl butterfly}. De Bonis, Katona, and Swanepoel \cite{debonis} proved $La(n,{\cal O}_4) = \sum(n, 2)$.
Griggs and Lu \cite{griggslu} proved $\pi({\cal O}_{2t})=e({\cal O}_{2t})$ for $t\geq 4 $ even, and recently Lu \cite{lu} 
proved $\pi({\cal O}_{2t})=e({\cal O}_{2t})$ for $t\geq 7 $ odd, leaving open only ${\cal O}_{6}$ and ${\cal O}_{10}$.

The most famous open problem about excluded posets  is that of the diamond  $D_2$. Griggs and Lu
showed
\begin{equation} \label{diamondbounds}
2\leq \liminf_{n\rightarrow \infty}
\frac{La(n,D_2)}{\binom{n}{\lfloor n/2\rfloor}}\leq \limsup_{n\rightarrow \infty}
\frac{La(n,D_2)}{\binom{n}{\lfloor n/2\rfloor}}\leq 2.296.
\end{equation}
The cited conjecture of Griggs and Lu would imply that in (\ref{diamondbounds}) the limit exists and is equal to 2. This is what we refer to as  the {\sl diamond conjecture}.
Axenovich, Manske and Martin \cite{axen}  reduced the upper bound  in  (\ref{diamondbounds})    to 2.283; Griggs,
Li, and Lu \cite{griggslilu} further reduced it to $25/11$.
The current best upper bound is 2.25, achieved by  Kramer, Martin, and Young \cite{martin}. They also pointed out that this is the  best possible bound that can be
derived from a Lubell function argument. Manske and Shen \cite{3layer} has a better upper bound for 3-layered families of sets, 2.1547, improving on an earlier bound of 
Axenovich, Manske and Martin \cite{axen},  2.207. A similar improvement  
 for 3-layered families of sets was also made by Balogh et al.  \cite{balogh}.

For a long time no construction was known  for the diamond problem with more
sets than those in the two largest levels, though  some alternative constructions existed,  e.g. taking on 12 points  all 5-subsets, all 7-subsets and a Steiner system
$S(5,6,12)$. In 2013, Andrew Dove \cite{dove} made an improvement on this, for every even $n\geq 6$.
 For $n=6$, he provided 36 sets,
while the two largest levels contain only 35. As $n$ goes to infinity, the gain in his construction is diminishing in percentage. 
Through computer search, Linyuan Lu independently found a  construction for $n=6$ with 36 sets.


The goal of this paper is to provide infinitely many exotic examples that show the  asymptotical tightness of the diamond conjecture. These constructions are based on Abelian groups and are very different from the usual extremal set systems. 
We show that Dove's example fits into this description although his formulation was different.
The proofs use the theory of Markov chains on groups, allowing citations of theorems instead of making analytic proofs from scratch
to show a limiting uniform distribution.   Dove's example, however, uses a non-uniform distribution.

\section{Strongly diamond-free Cayley posets} \label{cayley}
Let us be given a finite group $\Gamma$ with identity $e$ and a set of generators $H\subseteq \Gamma$. Recall that the {\sl Cayley graph}
$\vec G(\Gamma, H)$ has vertex set $\Gamma$ and edge set $\{g \rightarrow gh: h\in H,g\in \Gamma \}$. We do not assume $H=H^{-1}$, an assumption often made 
for Cayley graphs. 
 We define the {\sl infinite Cayley poset} $P(\Gamma, H)$ as follows:\\
the vertices of the poset are ordered pairs $(\gamma, i)$, for $\gamma\in\Gamma$ and $i\in \ZZ$, and \\
$(\gamma, i)\preceq_P (\delta,j)$,  if  $j\geq i$ and $\gamma=\delta \eta_1\eta_2\cdots \eta_{j-i}$ for some $\eta_1,\eta_2,\ldots ,\eta_{j-i}\in H$. It is easy to see that
$P(\Gamma, H)$  is a partial order indeed. Furthermore, mapping the vertices of the infinite Cayley poset $P(\Gamma, H)$ to the vertices of the Cayley graph $\vec G(\Gamma, H)$
by projection to the first coordinate, upward oriented edges of the Hasse diagram map to the edges of the Cayley graph. We term finite subposets of the infinite Cayley poset as
{\sl Cayley posets}.


 We say that $H$ is {\sl aperiodic}, if for
$L=\{\ell: \exists \eta_1,...,\eta_\ell\in H \hbox{ such that } \eta_1\eta_2\cdots \eta_\ell=e \}$, the greatest 
common divisor of elements of $L$ is 1. We say that a (finite) Cayley poset is {\sl aperiodic}, if the generating set is aperiodic.


Assume now that $\Gamma$ is abelian of order $m$ and $|H|=h$ with $H=\{\eta_1,\eta_2,\ldots, \eta_h\}$.
 For convenience,  for abelian groups we use the additive notation. Let us be given an $n$-element set $N$ partitioned
into classes $N_1,N_2,...,N_h$, such that $|N_i|=n_i$.  
 Assign for $x\in N$ a {\sl weight} $w(x)\in H$, such that
for all $x\in N_i$, $w(x)= \eta_i$. 
We will refer to $N$ as a {\sl weighted set}.
For $A\subseteq N$, define $w(A)=\sum_{x\in A} w(x)$. For  every $i\geq 0$ and  $\gamma \in\Gamma$,  define $$s_\gamma(i):=\{A\subseteq N: |A|= \lfloor n/2\rfloor   +
i \hbox{ and } w(A)=\gamma\}.$$

Let  $(\gamma_1, i_1)\prec (\gamma_2, i_2)\prec(\gamma_3, i_3)$ be three distinct elements of a Cayley poset $\Pi$. If some $\eta \in H$ can be used in both an $i_2-i_1$ term sum of elements of $H$ representing $\gamma_2-\gamma_1$ and an $i_3-i_2$ term sum of elements of $H$ representing $\gamma_3-\gamma_2,$ we say that the three elements form a \emph{strong chain} in $\Pi$. We call a Cayley poset $\Pi$ {\sl strongly diamond-free}, if (1) $\Pi$ is diamond-free, and (2)  it has no strong chains.  We need the following easy lemma:
\begin{lemma} \label{corresp}
If a  Cayley poset  $\Pi$ with elements  $\{(\gamma_i, j_i): i=1,2,...,\ell\}$ is strongly diamond-free, then
  for a weighted $n$-element set $N$,
 the family of sets  $$\FF(N,w,\Pi):= \bigcup_{(\gamma_i,j_i)\in \Pi}  s_{\gamma_i}(j_i)  =\{A\subseteq N: |A|={ \lfloor n/2\rfloor  +j_i} \hbox{ and } w(A)=\gamma_i,  \hbox{ for }   i=1,2,...,\ell \}$$ is diamond-free.
\end{lemma} 
\begin{proof}
Referring to a diamond in this proof,   we assume that
$a_1$ is its lowest element, $a_4$ is its largest element, and
$a_2,a_3$ are the middle (incomparable) elements.

We will show that if 
$  \FF(N,w,\Pi) $
 is not diamond-free,
then $\Pi$ is not strongly diamond-free.

If there are four different sets $A_1,A_2,A_3,A_4$ in $  \FF(N,w,\Pi) $
that correspond to a diamond $a_1,a_2,a_3,a_4$ resp., then,
$j_1< j_2,j_3< j_4$ and $j_i=|A_i|- {\binom{n}{\lfloor n/2\rfloor}}$, and also $\gamma_i=w(A_i)$.
Now we have that either ($j_2\ne j_3$) or
($j_2=j_3$ and $\gamma_2\ne\gamma_3$) or
$(j_2=j_3$ and $\gamma_2=\gamma_3)$

If  ($j_2\ne j_3$) or
($j_2=j_3$ and $\gamma_2\ne\gamma_3$), then the
four elements $(\gamma_i,j_i)$ $i\in[4]$ form a diamond
in  $\Pi$  so
$\Pi$ 
is not strongly diamond-free.

When $j:=j_2=j_3$ and $\gamma:=\gamma_2=\gamma_3$,
then
we have that
$w(A_1)=\gamma_1$, $w(A_4)=\gamma_4$, $|A_1|={\binom{n}{\lfloor n/2\rfloor}}+j_1$,
$|A_4|={\binom{n}{\lfloor n/2\rfloor}}+j_4$
$w(A_2)=w(A_3)=\gamma$ and $|A_2|=|A_3|={\binom{n}{\lfloor n/2\rfloor}}+j$.

Now clearly for $i\in\{2,3\}$ we have that
$A_1\subseteq A_2\cap A_3\subsetneq A_i\subsetneq A_2\cup A_3\subseteq A_4$.

Using $(A_2\cap A_3)\setminus A_1=\{x_1,\ldots,x_s\}$ (possibly empty), 
$A_4\setminus(A_2\cup A_3)=\{y_1,\ldots,y_k\}$ (possibly empty) and
$A_i\setminus (A_2\cap A_3)=\{z_1^i,z_2^i,\ldots,z_r^i\}$ (nonempty!) for
$i\in\{2,3\}$ it follows that
for $j\in(\{2,3\}\setminus\{i\})$ we have that
$(A_2\cup A_3)\setminus A_j=\{z_1^i,z_2^i,\ldots,z_r^i\}=
A_i\setminus(A_2\cap A_3)$.

It follows that $j=j_1+s+r$ and $j_4=j+r+k=j_1+2r+s+k$ and
that
$$\gamma=\gamma_1+\left(\sum_{\ell=1}^s w(x_{\ell})\right)
+\left(\sum_{\ell=1}^r w(z^2_{\ell})\right)$$
and
$$\gamma_4=\gamma+\left(\sum_{\ell=1}^k w(y_{\ell})\right)
+\left(\sum_{\ell=1}^r w(z^2_{\ell})\right),$$ where
$j-j_1=s+r$ and $j_4-j=k+r$.

Now since $r\ge 1$, we can choose $\eta:=w(z_1^2)$.

In particular, in this case we have found
$(\gamma_1,j_1)<(\gamma,j)<(\gamma_4,j_4)$ in  $\Pi$ 
such that $\gamma-\gamma_1$ can be written as a $(j-j_1)$-term
sum of elements of $H$ containing the term $\eta$ and
and $\gamma_4-\gamma$ is a $(j_4-j)$-term sum
of elements of $H$ containing the term $\eta$. Thus, in this case
$\Pi$ 
 is not strongly diamond-free either.
\old{
Now if in addition, we have that
the expected size of $S_{\gamma}(i)$ is
$\frac{1}{|G|}\cdot\binom{N}{\lfloor N/2\rfloor +i}=\frac{1}{|G}\binom{N}{\lfloor N/2\rfloor}(1+o(1))$ (using the fact that
${\mathcal G}$ is finite), then,
as $N\rightarrow\infty$,  we get that
$$\frac{|{\mathcal G}|}{|G|}\le
\liminf_{n\rightarrow\infty}
\frac{|{\mathcal S}_{{\mathcal G}}|}{\binom{N}{\lfloor N/2\rfloor}}
\le
\limsup_{N\rightarrow\infty}
\frac{|{\mathcal S}_{{\mathcal G}}|}{\binom{N}{\lfloor N/2\rfloor}}
\le \frac{La(n, diamond)}{\binom{N}{\lfloor N/2\rfloor}}$.
}
\end{proof} 

    \section{Markov chains on $\Gamma$} \label{markov}
Let us be given the set $N=[1,n]$. Assign  i.i.d. $H$-valued random variables $\omega(x)$ for every $x\in N$. Assume $\Pr[\omega(x)=\eta] >0$ for all $\eta\in H$ and extend the
probability distribution to $\gamma\in \Gamma\setminus H$ by $\Pr[\omega(x)=\gamma] =0$.   For an arbitrary $A\subseteq N$, assume $A=\{a_1<a_2<\cdots <a_{|A|}\}$.
Now we associate a finite Markov chain $X^A_j$ on $\Gamma$  for $j=0,1,...,|A|$ with $A$: define it with   $X_0^A=0$ for sure, and
$$X_i^A=\gamma   \hbox{\ iff \ }   \exists \delta\in \Gamma \hbox{\ such that } X_{i-1}^A=\delta 
\hbox{\ and \ }  \omega(a_{i})=\gamma-\delta     .$$
More explicitly, $X_i^A=\omega(a_1)+\omega(a_2)+\ldots + \omega(a_i)$. 
Consequently
$$\Pr[X_j^A=\gamma]= \sum_{\delta\in\Gamma}  \Pr[X_{j-1}^A=\delta ]\cdot \Pr[\omega(a_j)=\gamma-\delta]    .$$
If we defined analogously  the infinite Markov chain $X^A_j$ on $\Gamma$  for $j=0,1,...,$ for an infinite $A\subseteq \NN$,  the Markov chain would be irreducible 
if and only if $H$ is a generating set, and in this case the Markov chain would be aperiodic if and only if $H$ is not contained in a coset of a proper normal  subgroup of $\Gamma$ 
(see Proposition 2.3 in  \cite{saloff}). If the Markov chain is irreducible and aperiodic, then it converges to 
the unique stationary distribution on $\Gamma$, which is the uniform distribution 
(see p. 271 in  \cite{saloff}.) 
Hence assuming that $H$ is an aperiodic generating set,   it gives rise to an aperiodic Markov chain, and   $X^A_j$ converges to 
the uniform distribution.
The same results hold as well for  $X^A_j$ for a {\sl finite} set $A$, if $|A|$ is  {\sl sufficiently large}   for a fixed $\Gamma$.

\old{
the Fundamental Theorem of Markov
Chains \cite{behrends} that $\Pr[X_{|A|}^A=\gamma]$ is close  to  a unique stationary distribution. 
Furthermore, 
as the uniform distribution is stationary, the unique stationary distribution must be the uniform distribution. 
}
The Markov chains $X^A$  with different $A$'s do correlate, but we only will use the linearity of expectation.
Define $\omega(A)=\sum_{x\in A} \omega(x)$. For a fixed $i$ and a large $n$, set $$S_\gamma(i)=\{A\subseteq N: |A|= \lfloor n/2\rfloor   +
i \hbox{ and } \omega(A)=\gamma\},$$
a random family of sets.  Note that $\omega(A)=X_{|A|}^A  $ and $S_\gamma(i)$ are {\em random variables}, unlike $w(A)$ and $s_\gamma(i)$ in the previous section. By the 
convergence to uniform distribution recalled above, we have that for all $\epsilon>0$, for all sufficiently large $n$
\begin{equation} \label{epsilon}
\forall \gamma\in \Gamma \, \, \,  \forall A \, \, \,  \, \, \,   \frac{1}{|\Gamma|} -\epsilon < \Pr[\omega(A)=\gamma]    <  \frac{1}{|\Gamma|} +\epsilon. 
\end{equation}
Hence 
\begin{equation} \label{expectation}
\forall \gamma\in \Gamma \, \, \,    \frac{1}{|\Gamma|} -\epsilon < \frac{\E[|S_\gamma(i)|]}{{\binom{n}{\lfloor n/2\rfloor +  i}}}    <  \frac{1}{|\Gamma|} +\epsilon. 
\end{equation}
We reformulate (\ref{expectation}) above as a theorem:
  
\begin{theorem}\label{equidistr} {\rm  [Equidistribution theorem]}\\
Assume that $H$ is an aperiodic generating set of a finite abelian group of order $m$. Under the model above, for $i$ fixed as $n\rightarrow \infty$, we have
$$\lim_{n\rightarrow \infty} \frac{\E[|S_\gamma(i)|]}{\binom{n}{\lfloor n/2\rfloor +  i}}=\frac{1}{m}. $$
\end{theorem}
Observe that $ |\FF(N,w,\Pi)| \leq La(n,D_2)$, when $\Pi$ is strongly diamond-free.
Combining  Lemma~\ref{corresp}, Theorem~\ref{equidistr}, and the fact that for $i$ is fixed, the asymptotic formula
\begin{equation} \label{binomeq}
{\binom{n}{\lfloor n/2\rfloor +  i}}\sim {\binom{n}{\lfloor n/2\rfloor }}
\end{equation}
holds as $n\rightarrow \infty$,
 we immediately obtain the following theorem: 
\begin{theorem}\label{main} Assume that $H$ is an aperiodic generating set of a finite abelian group of order $m$.
If a fixed Cayley poset  $\Pi$ with elements  $(\gamma_i, j_i): i=1,2,...,\ell$ is strongly diamond-free, then 
$$\frac{\ell}{m} = \lim_{n\rightarrow \infty} \frac{\E[ |\FF(N,\omega,\Pi)|]  }{\binom{n}{\lfloor n/2\rfloor }} \leq\liminf_{n\rightarrow \infty} \frac{ La(n,D_2) }{\binom{n}{\lfloor n/2\rfloor }}.  $$

\end{theorem}
The conclusion of this theorem is that if one constructs an  aperiodic generating set of a finite abelian group of order $m$ and strongly diamond-free Cayley poset of $\ell$
elements with this generating set, then for a large $n$, an $\FF(N,w,\Pi)$ with some weighting $w$  has size at least $\Bigl( \frac{\ell}{m}-\epsilon\Bigl){\binom{n}{\lfloor n/2\rfloor }}$.
A construction with  $\ell >2m$ would even refute the diamond conjecture.

Unfortunately, the $\ell/m$ lower bound of Theorem~\ref{main} never exceeds two. The proof is the 
following. Take any $\eta \in H$ and partition the infinite Cayley poset $P(\Gamma, H)$ into $|\Gamma|$ chains as
$$\Biggl\{  \{(\gamma+i\eta, i): i\in \ZZ\}  : \gamma\in \Gamma    \Biggl\}.$$
If a finite Cayley poset $\Pi$ is free of strong chains, which is part of the requirement to be strongly 
diamond-free, then $\Pi$ cannot have more than two elements from any of the  
 $\{(\gamma+i\eta): i\in \ZZ\} $, for any $\gamma$.

Therefore in the next section we focus on constructing finite Cayley posets with $\ell=2m$ or with just slightly fewer elements. 

\old{
 \section{The main result in the periodic case}
Reconsider the Markov chain defined in Section~\ref{markov} without the assumption of
aperiodicity.
Note that if a state is periodic with period $d$ for a Markov Chain on a finite abelian group, then all states are periodic
with period $d$. The reason is that for all of them the period is gcd$(L)$ as $L$ is defined in Section~\ref{cayley}. Let $\sigma_j$ denote the cardinality of $j\times H:=H+H+\ldots+H$, a $j$-fold sum.
It is clear that $\sigma_j\leq \sigma_{j+1}$, and therefore after a sufficiently large $j$, $\sigma_j$ 
turns a constant $\sigma=\max_j \sigma_j$.  Assume now gcd$(L)=d>1$.
We want to understand better the structure of $j\times H$. 
\begin{claim} \label{acclaim}
If $\gamma\in j\times H$, then for some $t_0>0$, for all $t\geq t_0$ integers 
$\gamma\in (j+td)\times H$.
\end{claim}
\begin{proof}
Assume that  $\ell_1,...,\ell_t \in L$ and gcd$(L)=$gcd$(\ell_1,...,\ell_t)$.
A classic result of Frobenius asserts that if $d=$gcd$(\ell_1,...,\ell_t)$ for some $\ell_1,...,\ell_t$ positive 
integers, then for every sufficiently large positive integer $r$, $rd=\sum_{i=1}^t x_i\ell_i$, with some
$x_i\geq 0$ integers. Therefore for any $0\leq j\leq d-1$,
$$j\times H\subseteq  (j\times H)+ (\ell_1x_1\times H)+\ldots + (\ell_tx_t\times H) =(j+rd)\times H,$$
as $0\in \ell_i\times H$ and hence  $0\in \ell_ix_i\times H$.
\end{proof}
Fixing $j$ and increasing $r$, the sets $(j+rd)\times H$ stays the very same set
as $r$ passes a certain threshold. 
For any $0\leq j\leq d-1$,
we denote this set by $\tilde{H}_j$. From the arguments above,
for every $j$,  $|\tilde{H}_j|=\sigma$.
Now we select a subsequence from the  Markov chain $X^A$ on $\Gamma$  
defined in Section~\ref{markov}, such that this subsequence will be Markov chain again.
Fix any $0\leq i\leq d-1$. Let $X_i^A$ be the $i^{th}$ state of $X^A$, a random variable. For $r\geq 1$
observe that 
$$X_{i+rd}^A=\gamma   \hbox{\ iff \ }   \exists \delta\in \Gamma \hbox{\ such that } X_{i+(r-1)d}^A=\delta 
\hbox{\ and \ }   \sum_{j=(r-1)d+i+1}^{rd+i} \omega(A^{(j)})=\gamma-\delta     .$$
Clearly $Y_r^{(i)}:=X_{i+rd}^A$ is a Markov chain on the state space $\tilde{H}_i$. 
We show for every $i$  the irreducibility of the Markov chain
$Y_r^{(i)}$  (for sufficiently large $A$). 
Assume that $B\subset \tilde{H}_i$ is a proper subset, from which
one never obtains an element of $ \tilde{H}_i\setminus B$ by adding to it $qd$ elements of $H$ for any 
positive integer $q$. In other words, $B+(qd\times H) \subseteq B$ for every positive integer $q$. Hence $B+\tilde{H}_0
\subseteq B\subseteq \tilde{H}_i$.  As $|\tilde{H}_0|=| \tilde{H}_i|$, this implies $|B|=|\tilde{H}_i|$ and $B=\tilde{H}_i$, a contradiction.
}
\old{
First we do this for $i=0$. Assume that $B\subset \tilde{H}_0$ is a proper subset, from which
one never obtains an element of $ \tilde{H}_0\setminus B$ by adding to it $qd$ elements of $H$.
In other words, $B+(qd\times H) \subseteq B$ for every positive integer $q$. Hence $B+\tilde{H}_0
\subseteq B\subseteq \tilde{H}_0$. This implies $|B|=|\tilde{H}_0|$ and $B=\tilde{H}_0$, a contradiction.
Assume now $0<i<d-1$. Take arbitrary $a,b\in \tilde{H}_i$. We have to show that we can obtain $b$ 
by adding some $pd$ elements of $H$ to $a$, for some positive integer $p$. Fix any $0\not=h\in H$.
It is easy to see that for $x\in \tilde{H}_0$, $x+ih\in \tilde{H}_i$, and $x\mapsto x+ih$ is bijective.
Therefore there is a $b^*\in \tilde{H}_0$ such that $b^*+ih=b$. By a similar argument, 
$a+(d-i)h\in \tilde{H}_0$. We use that wea already proved the claim for $i=0$. Hence
$b^*=\bigl(a+(d-i)h   \bigl)+\sum_{f=1}^{pd} h_f$, with $h_f\in H$. Finally,
$$b=ih +b^*=ih+\bigl(a+(d-i)h   \bigl)+\sum_{f=1}^{pd} h_f,$$
which is $a$ plus a sum of multiple-of-$d$ elements of $H$.
}
\old{
The aperiodicity of
$Y_r^{(i)}$ for every $i$ follows from the fact that gcd$(L)=d$.   Indeed, if a state of   $Y_r^{(i)}$ returns in $p$ steps, it means $0\in pd\times H$, i.e. $pd\in L$. As  gcd$(L)=d$, the gcd of all such $p$'s must be 1.  
Hence $Y_r^{(i)}$ converges to a unique stationary distribution on
$\tilde{H}_i$.  As the uniform distribution on $\tilde{H}_i$ is also stationary,    it follows that
$Y_r{(i)}$ converges to the uniform distribution, for large $r$ the probability every state of $\tilde{H}_i$
approaches to $1/\sigma$ (and elements from $\Gamma\setminus \tilde{H}_i$ never occur as values of $Y_r^{(i)}$).
The same way as in the aperiodic case, we have
$$\frac{La(n,D_2)}{\binom{n}{\lfloor n/2\rfloor }}\geq   \frac{\E[ |\FF(N,\Pi)|]  }{\binom{n}{\lfloor n/2\rfloor }}=
\sum_{i=1}^\ell \frac{\E[|S_{\gamma_i}(j_i)|]}{\binom{n}{\lfloor n/2\rfloor }}.$$
For positive integers $a,b$,
 let $a {\rm \ modulo \ } b$ denote the remainder when $a$ is divided by $b$, and the remainder is
 between $0$ and $b-1$.  Fix a $0\leq i \leq \ell$.
Set $t=     \lfloor n/2\rfloor +j_i  {\rm \ modulo \ } d  $.  Observe that if $\gamma_i\notin \tilde{H}_t$, then 
$S_{\gamma_i}(j_i)=\emptyset$. On the other hand, if $\gamma_i\in \tilde{H}_t$, then 
$$
\frac{\E[|S_{\gamma_i}(j_i)|]}{\binom{n}{\lfloor n/2\rfloor +j_i}}
$$
is near $1/\sigma$. Using the argument of (\ref{binomeq}) again, we conclude
\begin{theorem} \label{main3} If a Cayley poset  $\Pi$ with elements  $(\gamma_i, j_i): i=1,2,...,\ell$ is strongly diamond-free, then for any $0\leq p\leq d-1$, we have
\begin{equation} \label{mainRHS}
 \liminf_{n\rightarrow \infty \atop \lfloor n/2\rfloor \equiv p {\rm \ mod \ } d} \frac{La(n,D_2)}{\binom{n}{\lfloor n/2\rfloor }}\geq  \frac{1}{\sigma}  \Biggl|\Biggl\{1\leq i\leq \ell:   \gamma_i\in 
\tilde{H}_ {j_i+p {\rm\  modulo\ } d} \Biggl\}\Biggl|.
\end{equation}
\end{theorem}
Unfortunately, the RHS of (\ref{mainRHS}) is at most two. Indeed, fix $0\not=h\in H$, and  for every  $g\in \tilde{H}_p$, consider the 
infinite chains $g+jh: j\in \NN$. Every such chain can contribute by at most 2 to the count, and 
as $| \tilde{H}_p|=\sigma$, the RHS is at  most $\frac{1}{\sigma}2\sigma =2$.
XXX 
Between XXX's, comments to Aaron and Travis:
ACTUALLY, somewhat less is needed than strong diamond-freeness of the finite poset - but 
hard to describe.
It would be nice to have more tight examples. Can we make examples with the periodic case
where the finite Cayley poset does not have consecutive levels? 
For illustration to you, we worked out an example in the periodic case.
Take the construction of Example 2 with $m=15$, $a=6$, $b=11$. The period is $d=5$.
$\sigma=15/5=3$.
It is easy to see:
$\tilde{H}_1=\{ 1,6,11 \}$, 
$\tilde{H}_2=\{ 2,7,12 \}$, $\tilde{H}_3=\{ 3,8,13 \}$, $\tilde{H}_4=\{ 4,9,14 \}$, $\tilde{H}_0=\{ 0,5,10 \}$.
The set whose cardinality is taken in the RHS in 
(\ref{mainRHS}) is 
$p=0$ $\{5,10; 6,11;7,12\}$
$p=1$ $\{1,6,11; -;3,8,13\}$
$p=2$ $\{2,7,12;-;4,9,14\}$
$p=3$ $\{3,8,13;-;0,5,10\}$
$p=4$ $\{4,9,14; -; 1,6,11\}$. 
The significance of the $;$ in the set notation is that it separates 
items corresponding to the first, second and third row (counted from below) of Example 2.
XXX
}

\section{Constructions} \label{constructions}

\begin{example}
 {\rm [The classic example.]}
For any   $\Gamma$  and $H$, take two levels from the infinite Cayley poset. This is a strongly diamond-free Cayley poset with $2m$ vertices.
\end{example}
We leave the verification of the  correctness of the following constructions to the readers, where $\Gamma$ denotes the additive group of modulo $m$ residue classes.
\begin{example}
For an odd $m$, take $\Gamma=\ZZ_m$ with $H=\{a,b\}$ such that $gcd(a,b)=1$. The following is a strongly diamond-free Cayley poset with $2m$ vertices:\\
$(g, 3): g\not\equiv a+b \hbox{ mod } m$,\\ $(a,2)$, $(b,2)$,\\   $(g, 1): g\not\equiv 0 \hbox{ mod } m$. 
\end{example}
(Note that if $gcd(a,b)=1$, then $H=\{a,b\}$ is a set of generators.)
This poset if often aperiodic,  for example, $a=1, b=2$ are such for $m=3$. 
The exact condition for aperiodicity is that $H$ is not contained by a coset of a proper  subgroup of $\Gamma$.
(See Proposition 2.3 in  \cite{saloff}.)

\begin{example}
Take $\Gamma=\ZZ_7$ with $H=\{2,3,5\}$. The following is a strongly diamond-free, aperiodic Cayley poset with $13$ vertices:\\
$(g, 3): g\not\equiv 0,1,5 \hbox{ mod } 7$,\\ 
$(2,2)$, $(3,2)$,  $(5,2)$\\   
$(g, 1): g\not\equiv 0 \hbox{ mod } 7$.
\end{example}
Note that 1,2 mod 3 and 2,3,5 mod 7 are difference sets. However, bigger difference sets do not seem to offer good constructions.
On four levels, we still can construct "close" constructions.
\begin{example}
For $m=4k-1$, take $\Gamma=\ZZ_m$ with $H=\{2k-1,2k\}$, $k\geq 2$. The following is a strongly diamond-free, aperiodic Cayley poset with $2m-2$ vertices:\\
$(i, 4): i=k+2,...,3k-3$,\\ 
   $(i, j): i=k,k+1,...,3k-1$; for  $j=1,2,3$.
\end{example}
\begin{example}
For $m=4k+1$, take $\Gamma=\ZZ_m$ with $H=\{2k,2k+1\}$, $k\geq 2$. The following is a strongly diamond-free, aperiodic Cayley poset with $2m-2$ vertices:\\
$(i, 4): i=k+2,...,3k-2$,\\ 
   $(i, j): i=k,k+1,...,3k$; for  $j=1,2,3$.
\end{example}

The last two constructions still allow close approximations of the conjectured maximum $(2+o(1)){\binom{n}{\lfloor n/2\rfloor }}$. For any fixed $\epsilon$, set $k>1+\frac{1}{\epsilon}$
to have  $\frac{2m-2}{m}>2-\frac{\epsilon}{2}$. Fixing this $k$, for sufficiently large $n$, a set system is obtained from Example 4 or 5  with at least  $(2-\epsilon ){\binom{n}{\lfloor n/2\rfloor }}$
elements.

Andrew Dove's construction  for even $n\geq 6$ is the following: take for the underlying set $N=\{1,2,...,n\}$, 
select those $\lfloor \frac{n}{2} \rfloor-1$ element sets that
do not contain the set $\{1,2\}$, 
select those  $\lfloor \frac{n}{2} \rfloor$ element sets that {\it do not} contain exactly one element from the set $\{1,2\}$, and 
  select those $\lfloor \frac{n}{2} \rfloor+1$ element sets that
have non-empty intersection with the set $\{1,2\}$. It is easy to check that this family is diamond-free. Say, for $n=6$, the family has
$$\Biggl[\binom{6}{2}-1\Biggl]+\Biggl[\binom{4}{3}+\binom{4}{1}\Biggl]+\Biggl[\binom{6}{4}-1\Biggl]=14+8+14=36$$
elements, while two distinct middle levels have only $\displaystyle \binom{6}{2}+ \binom{6}{3}=15+20=35 $ elements.

\begin{example}
Take $\Gamma=\ZZ_3$ with $H=\{0,2\}$. The following is a strongly diamond-free, aperiodic Cayley poset with $6$ vertices:\\
$(g, \phantom{} 1): g\not\equiv 1 \hbox{ mod } 3$,\\ 
$(g, 0): g\not\equiv 0 \hbox{ mod } 3$,\\   
$(g, -1): g\not\equiv 2 \hbox{ mod } 3$.
\end{example}
Andrew Dove's construction can be obtained from this Example 
using Lemma \ref{corresp} but not Markov chains.
Take for the underlying set $N=\{1,2,...,n\}$, set $N_1=\{1,2\} $ with weights $w(1)=w(2)=2\in \ZZ_3$ and $N_2=\{3,4,...,n\}$
with weights $w(3)=w(4)=\ldots =w(n)=0\in \ZZ_3$. This construction gives  the very same family as Andrew Dove's construction.


\noindent{\bf Acknowledgement.} The authors thank the referees for their comments and corrections.

\end{document}